\documentclass[12pt]{article}
\input isolatin1.sty
 \usepackage{a4}
 \usepackage{amsmath}
 \usepackage{amssymb}
\usepackage{amsfonts}
\usepackage{amsthm}
\usepackage{graphicx}
\newtheorem{thm}{Theorem}[section]

\newtheorem{lem}[thm]{Lemma}
\newtheorem{prop}[thm]{Proposition}

\newtheorem{rem}[thm]{Remark}

\renewcommand{\a}{\alpha}
\renewcommand{\b}{\beta}
\newcommand{\la}{\lambda}

\renewcommand{\span}{\operatorname{span}}

\def\tr{\mbox{tr}\,}

\long\def\comment#1{{}}

\title{The smallest eigenvalue of Hankel matrices}

\author{Christian Berg\thanks{The present work was initiated while the first
     author was visiting University of Wroc{\l}aw granted by the HANAP
    project mentioned under the second author. The first author has
    been supported by grant 272-07-0321 from the Danish Research
    Council for Nature and Universe.}
\and Ryszard Szwarc
 \thanks{The second author was supported by European Commission Marie 
Curie Host
Fellowship for the Transfer of Knowledge ``Harmonic Analysis, Nonlinear 
Analysis and
Probability''  MTKD-CT-2004-013389 and by MNiSW Grant N201 054 32/4285.}
}
\begin{document}
\maketitle

\begin{abstract} Let $\mathcal H_N=(s_{n+m}),n,m\le N$ denote the Hankel matrix of
  moments of a positive measure with moments of any order. We study
  the large $N$ behaviour of the smallest eigenvalue $\la_N$ of
  $\mathcal H_N$. It is proved that $\la_N$ has exponential decay to
  zero for any measure with compact support. For general
  determinate moment problems the decay to 0 of $\la_N$ can be
  arbitrarily slow or arbitrarily fast. In the indeterminate case, where $\la_N$ is known
  to be bounded below, we prove that the limit of the $n$'th smallest eigenvalue of
  $\mathcal H_N$ for $N\to\infty$ tends rapidly to infinity with
  $n$. The special case of the Stieltjes-Wigert polynomials is discussed. 
  \end{abstract}


\section{Introduction}
Let $(s_n)$ be the moment sequence of a  positive measure $\mu$ on
$\mathbb R$ with infinite support,
\begin{equation}\label{eq:HMP}
s_n=\int x^n\,d\mu(x),\quad n\ge 0.
\end{equation}
 By Hamburger's theorem this is
equivalent to a real sequence $(s_n)$ such that all the Hankel
matrices
\begin{equation}\label{eq:Hankel}
\mathcal H_N=(s_{n+m})_{n,m=0}^N,\quad N=0,1,\ldots
\end{equation}
are positive definite. The smallest eigenvalue of $\mathcal H_N$ is the
positive number 
\begin{equation}\label{eq:lambda}
\la_N=\min\{\langle \mathcal H_N a,a\rangle \mid a\in\mathbb
C^{N+1},||a||=1\},
\end{equation}
and clearly $\la_0\ge\la_1\ge\ldots$. The large $N$ behaviour of
$\la_N$  has been studied in the papers
\cite{B:C:I,Ch:La,Ch:Lu,Sz,W:W,Wi}. See also results in \cite{Bec,Lu} about the behaviour
of the condition number $\kappa(\mathcal H_N)=\Lambda_N/\lambda_N$,
where $\Lambda_N$ denotes the largest eigenvalue of $\mathcal H_N$.

Widom and Wilf \cite{W:W} found the asymptotic behaviour
\begin{equation}\label{eq:ww}
\lambda_N\sim AN^{1/2}B^{N},
\end{equation}
for certain constants $A>0,0<B<1$ in the case of a measure $\mu$ of
compact support in
the Szeg\H{o} class, generalizing results by  Szeg\H{o} \cite{Sz}. In
the same paper Szeg\H{o} also obtained results about the Hermite and
Laguerre case, namely
\begin{equation}\label{eq:sz}
\la_N\sim AN^{1/4}B^{N^{1/2}},
\end{equation}
again with certain $A,B$ as above. In all of this paper $a_N\sim b_N$ means
that $a_N/b_N\to 1$ as $N\to\infty$.
 
Chen and Lawrence \cite{Ch:La} found the asymptotic behaviour of
$\la_N$ for the case of $\mu$ having the density $e^{-t^\b}$ with
respect to Lebesgue measure on the interval $[0,\infty[$.
 The result requires
$\b>1/2$, and we refer to \cite{Ch:La} for the quite involved
expression. For $\b=\tfrac12$ the asymptotic behaviour is only
stated as a conjecture:
$$
\la_N\sim A\frac{\sqrt{\log N}}{N^{2/\pi}}
$$
for a certain constant $A>0$.

Chen and Lubinsky \cite{Ch:Lu} found the asymptotic behaviour of
$\la_N$, when $\mu$ is a generalized (symmetric) exponential weight
including $e^{-|x|^\a}$ with $\a>1$.

 We recall that the density $e^{-t^\b}$ on the half-line is determinate for $\b\ge
\tfrac12$, i.e. there are no other measures having the moments
\begin{equation}\label{eq:beta}
s_n=\int_0^\infty t^ne^{-t^\b}\,dt=\Gamma\left(\frac{n+1}{\b}\right)/\b.
\end{equation}
However, for $0<\b<\tfrac12$ the density is Stieltjes indeterminate: There are
infinitely many measures on the half-line with the moments \eqref{eq:beta}.
The symmetric density $e^{-|x|^\a}$ is determinate if and only if
$\a\ge 1$.
 For general
information about the moment problem see \cite{Ak,S:T,Simon}.

 Berg, Chen and Ismail proved in \cite{B:C:I} the general result that the moment sequence \eqref{eq:HMP}
(or the measure $\mu$) is determinate
if and only if $\la_N \to 0$ for $N\to\infty$ and found the positive 
lower bound  $\la_N\ge 1/\rho_0$ in the indeterminate case, where
$\rho_0$ is given in \eqref{eq:rho} below.

The purpose of the present paper is to prove some general results
about the behaviour of $\la_N$.

In section 2  we prove that $\la_N$ tends to zero exponentially
for any measure $\mu$ of compact support. Theorem~\ref{thm:fast} is a
slightly sharpened version, where only the boundedness of the
coefficients $(b_n)$ from the three term recurrence relation \eqref{eq:3t}
is assumed.
We also show that $\la_N$ may tend to zero arbitrarily fast.

Section 3 is devoted to showing that there exist determinate measures
for which $\la_N$ tends to zero arbitrarily slowly, cf. Theorem~\ref{thm:slow}.
 
In Section 4 we consider the indeterminate case, where $\la_N$ is
bounded below by a positive constant. We prove that the $n$'th
smallest eigenvalue $\la_{N,n}$ of \eqref{eq:Hankel} ($n\le N$) has a
lower bound $\la_{\infty,n}=\lim_{N\to\infty}\la_{N,n}$,which tends rapidly to infinity with $n$, cf. Theorem~\ref{thm:lambdaNn}. To describe our
results in detail we need some more notation.  
  
We let $(P_n)$ denote the sequence of orthonormal polynomials with
respect  to $\mu$, uniquely determined by the requirements that $P_n$
is a polynomial of degree $n$ with positive leading coefficient and
the orthonormality condition $\int P_nP_m\,d\mu=\delta_{nm}$.

The orthonormal polynomials satisfy the following three-term
recurrence relation
\begin{equation}\label{eq:3t}
 xP_n(x)=b_nP_{n+1}(x)+a_nP_n(x)+b_{n-1}P_{n-1}(x),
\end{equation}
where $b_n>0$ and $a_n\in\mathbb R$.

We need the coefficients of the orthonormal polynomials $(P_n)$ with respect to $\mu$:
\begin{equation}\label{eq:OP1}
P_n(x)=\sum_{k=0}^nb_{k,n}x^k,
\end{equation}
and consider the infinite upper triangular matrix
\begin{equation}\label{eq:utmB}
\mathcal B=(b_{k,n}),\qquad b_{k,n}=0,\ k>n.
\end{equation}

Let $\mathcal B_N$ denote the $(N+1)\times(N+1)$-matrix
obtained from $\mathcal B$ by assuming $k,n\le N$ and let $\mathcal A^{(N)}=\mathcal B_N\mathcal
B_N^*$.
Defining the kernel polynomial
\begin{equation}\label{eq:kernelpol}
K_N(z,w)=\sum_{n=0}^N
P_n(z)P_n(w)=\sum_{j,k=0}^N\left(\sum_{n=\max(j,k)}^N
  b_{j,n}b_{k,n}\right)z^jw^k,
\end{equation}
we see that  $\mathcal A^{(N)}=(a_{j,k}^{(N)})$ is the $(N+1)\times(N+1)$-matrix of coefficients to $z^jw^k$ in
$K_N(z,w)$.
The following result going back to A.C. Aitken, cf. Collar
\cite{Co}, has been rediscovered several times, see \cite{Be2,Si}.

\begin{thm}\label{thm:ABC} 
$$
\mathcal A^{(N)}=\mathcal H_N^{-1}.
$$
\end{thm}  

For completeness we give the simple proof of Theorem~\ref{thm:ABC}:

For $0\le k\le N$ we have by the reproducing property
$$
\int x^kK_N(x,y)\,d\mu(x)=y^k.
$$
On the other hand we have
$$
\int x^kK_N(x,y)\,d\mu(x)=
\sum_{j=0}^N(\sum_{\ell=0}^N s_{k+\ell}a^{(N)}_{\ell,j})y^j,
$$
and therefore
$$
\sum_{\ell=0}^N s_{k+\ell}a^{(N)}_{\ell,j}=\delta_{k,j}. \quad\square
$$

The following Lemma is also very simple. The identity matrix is
denoted $I=(\delta_{j,k})$.

\begin{lem}\label{thm:ABCinfty} As infinite matrices we have
$$
\mathcal B(\mathcal B^*\mathcal H)=(\mathcal B^*\mathcal H)\mathcal
B=\mathcal B^*(\mathcal H \mathcal B)=I,
$$
and $\mathcal B^*\mathcal H$ is an upper triangular matrix.
\end{lem} 

\begin{proof} The matrix products $\mathcal B^*\mathcal H$ and
  $\mathcal H \mathcal B$
  are well-defined because $\mathcal B$ is upper triangular, and we get
$$
(\mathcal B^*\mathcal H)_{j,k}=\sum_{n=0}^j b_{n,j}s_{n+k}=\int
P_j(x)x^k\,d\mu(x),
$$
which is clearly 0 for $j>k$,
so $\mathcal B^*\mathcal H$ is also upper triangular. Therefore,
$\mathcal B(\mathcal B^*\mathcal
H)$ is well-defined and upper triangular. For $l\le k$ we finally get
$$
(\mathcal B(\mathcal B^*\mathcal H))_{l,k}=\sum_{j=0}^k b_{l,j}\sum_{n=0}^j b_{n,j}s_{n+k}=\sum_{n=0}^k\left(\sum_{j=0}^kb_{l,j}b_{n,j}\right)s_{n+k}=\delta_{l,k}
$$
by Theorem~\ref{thm:ABC} with $N=k$.

The relation $(\mathcal B^*\mathcal H)\mathcal B=\mathcal B^*(\mathcal
H \mathcal B)=I$ is an easy
consequence of the orthogonality of $(P_n)$
with respect to $\mu$.
\end{proof}

 We also consider the infinite matrix
\begin{equation}\label{eq:kappa}
\mathcal K=(\kappa_{j,k}),\quad
\kappa_{j,k}=\frac{1}{2\pi}\int_0^{2\pi}P_j(e^{it})P_k(e^{-it})\,dt.
\end{equation}

 It is a
classical fact that the indeterminate case occurs if and only if
\begin{equation}\label{eq:indet}
\sum_{n=0}^\infty |P_n(z)|^2<\infty
\end{equation}
for all $z\in\mathbb C$. It suffices that \eqref{eq:indet} holds for
just one point $z_0\in\mathbb C\setminus\mathbb R$, and in this case
the convergence of  \eqref{eq:indet} is uniform on compact subsets of
the complex plane.

In the indeterminate case we can let $N\to\infty$ in
\eqref{eq:kernelpol} leading to the entire function of two
complex variables
\begin{equation}\label{eq:kernel}
K(z,w)=\sum_{n=0}^\infty P_n(z)P_n(w)=\sum_{j,k=0}^\infty a_{j,k}z^jw^k,
\end{equation}
and we collect the coefficients of the power series as the symmetric matrix
\begin{equation}\label{eq:ajk}
\mathcal A=(a_{j,k}).
\end{equation}

In Proposition~\ref{thm:matrix} we prove that the matrices $\mathcal
A,\mathcal B,\mathcal K$ are of trace class in the indeterminate case and 
$$
\tr(\mathcal A)=\tr(\mathcal K)=\rho_0,
$$
where $\rho_0$ is given by
\begin{equation}\label{eq:rho}
\rho_0=\frac{1}{2\pi}\int_0^{2\pi}K(e^{it},e^{-it})\,dt=
\frac{1}{2\pi}\int_0^{2\pi}\sum_{k=0}^{\infty}\bigl|P_k\left({\rm e}^{it}\right)
\bigr|^2\,dt<\infty.
\end{equation}

In the indeterminate case the infinite Hankel matrix $\mathcal
H=(s_{n+m})$ does not correspond to an operator on $\ell^2$ defined on
$\span\{\delta_n|n\ge 0\}$. In fact,  by Carleman's theorem we
necessarily have $\sum_{n=0}^\infty s_{2n}^{-1/(2n)}<\infty$, hence
$s_{2n}\ge 1$ for $n$ sufficiently large, and therefore
$$
\sum_{m=0}^\infty s_{n+m}^2=\infty\;\mbox{for all}\;n.
$$ 

It is likely that Theorem~\ref{thm:ABC} extends to the indeterminate
case in the sense that $\mathcal A\mathcal H=\mathcal H\mathcal A=I$,
where the infinite series $\sum_l a_{k,l}s_{l+j}$ defining  $\mathcal
A\mathcal H$ and $\mathcal H\mathcal A$ are absolutely convergent. We
have not been able to prove this general statement, but it holds for
the Stieltjes-Wigert case which is treated in Section 5.

The Stieltjes-Wigert polynomials $P_n(x;q)$ are defined in
\eqref{eq:SW2}. They are orthogonal with respect to a log-normal distribution, known to be
indeterminate, and the corresponding moment sequence is
$s_n=q^{-(n+1)^2/2}$. It is known that the modified moment sequence
$(\tilde s_n)$ given by $\tilde s_n=s_n$ for $n\ge 1$ and 
$$
\tilde s_0=s_0-(\sum_{n=0}^\infty P_n(0;q)^2)^{-1}
$$
is determinate, and the corresponding measure $\tilde{\mu}$ is
discrete given by
\begin{equation}\label{eq:index0}
\tilde{\mu}=\sum_{x\in X}c_x\delta_x,
\end{equation}
where $X$ is the zero set of the reproducing kernel $K(0,z)$ defined
in \eqref{eq:kernel} and 
\begin{equation}\label{eq:index1}
c_x=\left(\sum_{k=0}^\infty P_k(x;q)^2\right)^{-1},\quad x\in X.
\end{equation}
The Hankel matrices $\mathcal H=(s_{j+k})$ and $\tilde{\mathcal
  H}=(\tilde s_{j+k})$ agree except for the upper left corner.
 In Theorem~\ref{thm:SW2}   we prove that the smallest
eigenvalue $\tilde\la_N$ of the Hankel matrix $\tilde{\mathcal
H}_N$ tends to zero exponentially (while $\la_N$ is bounded below). We do it by determining the
corresponding orthonormal polynomials $\tilde P_n(x;q)$, see Theorem~\ref{thm:SW3}.

\section{Fast decay}

We start by proving a lemma which is essentially contained in
\cite[§2]{B:C:I}.

\begin{lem}\label{thm:upperbound} For each $z_0\in\mathbb C$ with
  $|z_0|<1$ we have
\begin{equation}\label{eq:upperbound}
\la_N\le\left((1-|z_0|^2)\sum_{n=0}^N |P_n(z_0)|^2\right)^{-1}.
\end{equation}
\end{lem}

\begin{proof} For any $a\in \mathbb C^{N+1}, a\neq 0$ we have by \eqref{eq:lambda}
$$
\la_N\le \frac{\langle \mathcal H_N a,a\rangle}{||a||^2}.
$$
This means that for any non-zero polynomial
$$
p(x)=\sum_{k=0}^N a_kx^k =\sum_{n=0}^N c_nP_n(x)
$$
we have
\begin{equation}\label{eq:up1}
\la_N\le
\frac{\int|p|^2\,d\mu}{\frac{1}{2\pi}\int_0^{2\pi}|p(e^{it})|^2\,dt}.
\end{equation}
Moreover, by Cauchy's integral formula
$$
p(z_0)=\frac{1}{2\pi}\int_0^{2\pi}\frac{p(e^{it})e^{it}}{e^{it}-z_0}\,dt,
$$
hence by Cauchy-Schwarz's inequality
\begin{equation}\label{eq:up2}
|p(z_0)|^2\le
\frac{1}{2\pi}\int_0^{2\pi}|p(e^{it})|^2\,dt\,\frac{1}{2\pi}\int_0^{2\pi}\frac{dt}{|e^{it}-z_0|^2},
\end{equation}
and the last integral equals $(1-|z_0|^2)^{-1}$ by a well-known
property of the Poisson kernel. Combining \eqref{eq:up1} and
\eqref{eq:up2} for the polynomial
$$
p(x)=\sum_{n=0}^N \overline{P_n(z_0)}P_n(x)
$$
leads to
$$
\la_N\le \frac{\sum_{n=0}^N
  |P_n(z_0)|^2}{(1-|z_0|^2)|p(z_0)|^2}=\left((1-|z_0|^2)\sum_{n=0}^N
  |P_n(z_0)|^2\right)^{-1}.
$$
\end{proof}

 \begin{rem} {\rm It follows immediately from
    Lemma~\ref{thm:upperbound} that if $\la_N\ge c>0$ for all $N$,
    then
$$
\sum_{n=0}^\infty |P_n(z_0)|^2<\infty
$$
for all $z_0$ with $|z_0|<1$, hence $(s_n)$ is indeterminate.}
\end{rem} 

 The following theorem proves that $\la_N$ tends to zero exponentially
 in the sense that there is an estimate of the form
\begin{equation}\label{eq:expfast}
\la_N\le A B^N,\quad A>0,0<B<1,
\end{equation}
whenever the measure $\mu$ in \eqref{eq:HMP} has compact support. 
 
\begin{thm}\label{thm:fast} Assume that the sequence $(b_n)$ from
  \eqref{eq:3t} is bounded with $b:=\limsup b_n$.
Then
$$
\limsup \la_N^{1/N}\le \frac{2b^2}{1+2b^2}.
$$
\end{thm}

\begin{rem} {\rm Notice that the condition $\limsup b_n < \infty$
    implies that $\sum 1/b_n=\infty$, so by Carleman's theorem the
    moment problem is determinate, cf. \cite[p.24]{Ak}. We also recall the fact that $\mu$
    has compact support if and only if $(a_n),(b_n)$ from \eqref{eq:3t} are bounded sequences.}
\end{rem}

\begin{proof} Taking  $z_0=\alpha i$, where $0<\alpha< 1$, we obtain
  from Lemma~\ref{thm:upperbound}
$$ \la_N \le \left((1-\alpha^2)\sum_{n=0}^N |P_n({\alpha
    i})|^2 \right)^{-1}\le \left((1-\alpha^2)[|P_{N-1}(\alpha i)|^2+|P_N(\alpha
  i)|^2]\right)^{-1}.
 $$
Since the distance from the point $\a i $ to the support of the
orthogonality measure is at least $\a$, we obtain
 by \cite[Remark 2, p. 148]{lb}   
$$
\limsup\lambda_N^{1/N}\le
\frac{1}{1+\frac{\a^2}{2b^2}}=\frac{2b^2}{\a^2+2b^2}.
$$
As $\a$ is an arbitrary number less than $1$ we get
$$
\limsup\lambda_N^{1/N}\le  \frac{2b^2}{1+2b^2}.
$$
\end{proof}

\begin{thm}\label{thm:fast1} For any decreasing sequence $(\tau_n)$ of
  positive numbers
  with $\tau_0=1$ and $\lim\tau_n=0$, there exist determinate
  probability measures $\mu$ for
  which $\la_N\le \tau_N$ for all $N$.
\end{thm}

\begin{proof} We will construct  symmetric probability measures $\mu$
  with the desired property. Let
\begin{equation}\label{eq:sym3t}
 xP_n(x)=b_nP_{n+1}(x)+b_{n-1}P_{n-1}(x)
\end{equation}
be the three-term recurrence relation for the orthonormal polynomials
associated with a symmetric $\mu$. 
We shall choose $b_n>0,n\ge 0$ such that $\la_N\le \tau_N$ for all
$N\ge 0$. We always have $\la_0=\tau_0=1$ because $\mu$ is a
probability measure. Since $s_1=0$ we know that
$\la_1=\min(1,s_2),s_2=b_0^2$, so we can choose $0<b_0\le 1$ such that $\la_1=\tau_1$.

By Lemma~\ref{thm:upperbound} with $z_0=0$ we get
$$
\la_N\le \left(\sum_{n=0}^N |P_n(0)|^2\right)^{-1},
$$
and in particular 
\begin{equation}\label{eq:fast1}
\la_{2N+1}\le \la_{2N}\le \frac{1}{P_{2N}^{2}(0)}.
\end{equation}
By \eqref{eq:sym3t} we have
 $$
P_{2n}(0)=(-1)^n \frac{b_0b_2\ldots b_{2n-2}}{b_1b_3\ldots
    b_{2n-1}},\quad n\ge 1,
$$ 
and defining 
$$
r_k=\frac{b_{2k-1}}{b_{2k-2}}, \quad k\ge 1
$$
we get
$$
\la_{2N+1}\le \la_{2N}\le r^2_1r^2_2\ldots r^2_N,\quad N\ge 1,
$$
and we will  choose $r_k,k\ge 1,$ such that 
$$
 r^2_1r^2_2\ldots r^2_N\le \tau_{2N+1},\quad N\ge 1.
$$
First choose $0<r_1\le\sqrt{\tau_3}$, and when $r_1,\ldots,r_{N-1}$ have
been chosen, we choose 
$$
0<r_N\le \min\left(1,\frac{\sqrt{\tau_{2N+1}}}{r_1\ldots r_{N-1}}\right).
$$
It is clear that the sequence $(r_k)$ can be chosen such that $r_k\to 0$.
We next define $b_1=r_1b_0$ and we finally have an infinity of choices of $b_{2k-1},b_{2k-2}>0$ to
satisfy $r_k=b_{2k-1}/b_{2k-2},k\ge 2$.

 If $(r_k)$ converges to zero, the decay of $\lambda_n$ is faster
than exponential. Clearly the corresponding moment problem is determinate
since 
$$
|P_{2n}(0)|=(r_1r_2\ldots r_n)^{-1}\ge 1.
$$
 In particular,   the unique measure   $\mu$ solving the moment problem
carries no mass at $0.$
\end{proof}

After having chosen the numbers $r_k$ we have several possibilities for
selecting the coefficients $b_n.$
We will discuss three  such choices.

 \medskip
\noindent{\bf Example 1.}
For $k\ge 2$ let $b_{2k-2}=1$ and $b_{2k-1}=r_k$ and assume that
$r_k\to 0.$ Then the corresponding Jacobi matrix
 $J$ is bounded and it acts on $\ell^2$ by
$$(Jx)_n=b_nx_{n+1}+b_{n-1}x_{n-1},\quad x=(x_n) .$$
Let us compute the square of $J$. We have
 $$(J^2x)_n=b_nb_{n+1}x_{n+2}+(b_{n-1}^2+b_{n}^2)x_n+b_{n-2}b_{n-1}x_{n-2}.$$
By the choice of $(b_n)$ we get $b_nb_{n+1}\to 0$ and $b_{n-1}^2+b_n^2\to 1.$ Therefore
the operator $J^2$ is of the form
$J^2=I+K$, where $K$ is a compact operator. Hence its spectrum consists of a sequence of positive numbers
converging to 1. Thus the spectrum of $J$ is of the form $\sigma(J)=\{\pm t_n\} $, where $t_n$ is a sequence of positive
numbers converging to 1, so the measure $\mu$ is discrete with bounded
support.

\medskip
\noindent{\bf Example 2.}
Let $b_{2k-2}=r^{-1}_k$ and $b_{2k-1}=1$ and assume $r_k\to 0.$ Then the corresponding Jacobi matrix $J$ is unbounded.  
 By the recurrence relation we have
\begin{equation}\label{eq:ex2eq1}
x^2P_{2n}(x)=b_{2n}b_{2n+1}P_{2n+2}(x)+(b_{2n-1}^2+b_{2n}^2)P_{2n}(x)+
b_{2n-2}b_{2n-1}P_{2n-2}(x).
\end{equation}
Then $Q_n(y)=P_{2n}(\sqrt{y})$ is a polynomial of degree $n$ satisfying
$$
yQ_n(y)=r_{n+1}^{-1}Q_{n+1}(y)+(1+r_{n+1}^{-2})Q_n(y)+
r_n^{-1}Q_{n-1}(y).
$$
Letting $B_n=r_{n}^{-1}$ and $A_n= (1+r_{n+1}^{-2})$ we get
$$
\frac{B_n^2}{A_{n-1}A_{n}}=\frac{r_{n+1}^2}{(1+r_n^2)(1+r_{n+1}^2)}\stackrel{n}{\longrightarrow}
0,
$$
so by   Chihara's Theorem (see \cite[Th. 8]{ch} and \cite[Theorem 2.6]{rs})  we see
that the orthogonality measure $\nu$ for $Q_n(y)$ is
discrete. However, $\nu$ is the image measure of the symmetric measure
$\mu$ under the mapping $x\to x^2$, so also $\mu$ is discrete with
unbounded support.

\medskip
\noindent{\bf Example 3.}
Let $b_{2k-2}=r^{-1/2}_k$ and $b_{2k-1}=r^{1/2}_k$. With
$Q_n(y)=P_{2n}(\sqrt{y})$ as in Example 2 we get from
\eqref{eq:ex2eq1}
$$
yQ_n(y)=Q_{n+1}(y)+a_nQ_n(y)+Q_{n-1}(y)
$$
where $a_n=r_n+1/r_{n+1}$. If $r_k\to 0$ we see again that $\mu$ is
discrete with unbounded support.

\section{Slow decay}

The goal of this section is to prove that there exist moment sequences
$(s_n)$ such that the corresponding sequence $(\la_N)$ from
\eqref{eq:lambda} tends to 0 arbitrarily slowly. This is proved in
Theorem~\ref{thm:slow}.
 
Consider a symmetric probability measure $\mu$ on the real line with
moments of any order and infinite support.
The corresponding orthonormal polynomials
$(P_n)$ satisfy
  a symmetric recurrence relation \eqref{eq:sym3t}, 
where  $b_n>0$ for $n\ge 0.$  For simplicity we assume
that the second moment of $\mu$ is 1, i.e.
 $s_2=b_0^2=1$ and hence $\la_0=\la_1=1$. This can always be achieved by replacing $d\mu(x)$ by
 $d\mu(ax)$ for suitable $a>0$. Note that $P_0=1,P_1(x)=x$ in this case.

\begin{lem}\label{thm:pni} Let $(P_n)$ denote the orthonormal
  polynomials satisfying \eqref{eq:sym3t} with $b_0=1$. The sequence 
\begin{equation}\label{eq:un}
u_n=|P_n(i)|,\quad n\ge 0
\end{equation}
satisfies
\begin{equation}\label{eq:2}
u_{n+1}=\frac{1}{b_n}\,u_n+\frac{b_{n-1}}{b_n}\,u_{n-1},\quad n\ge 1, 
\end{equation}
with $u_0=u_1=1$. Moreover, for $n\ge 0$ 
$$
|P_n(z)|\le u_n,\quad |z|\le 1. 
$$
\end{lem}
\begin{proof} Let $k_n=b_{n,n}$ denote the (positive) leading coefficient of $P_n$ and
  let $x_1,x_2\ldots,x_n$ denote
 the positive zeros of $P_{2n}.$ Then  
 $$
P_{2n}(x)=k_{2n}(x^2-x_1^2)(x^2-x_2^2)\ldots(x^2-x_n^2),
$$
hence 
\begin{equation}\label{eq:3}
u_{2n}=(-1)^nP_{2n}(i)>0.
\end{equation}
Similarly, let $y_1,y_2\ldots,y_n$ denote the positive zeros of $P_{2n+1}.$  
Then
$$
P_{2n+1}(x)=k_{2n+1}x(x^2-y_1^2)(x^2-y_2^2)\ldots(x^2-y_n^2),
$$
hence 
\begin{equation}\label{eq:4}u_{2n+1}=(-1)^{n+1}i\,P_{2n+1}(i)>0.
\end{equation}
Combining \eqref{eq:sym3t}, \eqref{eq:3} and \eqref{eq:4}  gives  \eqref{eq:2}.

By \eqref{eq:sym3t} we get for $|z|\le 1$ 
$$
|P_{n+1}(z)|\le
\frac{1}{b_n}\,|P_n(z)|+\frac{b_{n-1}}{b_n}\,|P_{n-1}(z)|.
$$
Therefore, \eqref{eq:2} can be used to
 show by induction 
that $|P_n(z)|\le u_n.$
\end{proof}

\begin{prop}\label{thm:pni2} Assume that the coefficients $(b_n)$ from
  \eqref{eq:sym3t} satisfy $b_0=1$ and $b_{n-1}+1\le  b_n,$ for $n\ge 1$
  and let $u_n=|P_n(i)|$. Then
$$
 \max(u_{2n},u_{2n+1})\le \prod_{k=1}^n\max\left
  (\frac{1+b_{2k-2}}{b_{2k-1}},\frac{1+b_{2k-1}}{b_{2k}} \right),\quad
n\ge 0.
$$
\end{prop}

\begin{proof}
Since
 $$
\frac{1+b_{k-1}}{b_k}\le 1,
$$
we get from \eqref{eq:2}
$$
u_{k+1}\le \max(u_{k-1},u_{k}),\quad k\ge 1.
$$
We have clearly
$$
u_k\le \max(u_{k-1},u_{k}),
$$
thus
$$
\max(u_{k},u_{k+1})\le \max(u_{k-1},u_{k}),\quad k\ge 1.
$$
This implies by \eqref{eq:2}
$$
u_{n+1}\le \frac{1+b_{n-1}}{b_n}\max(u_{n-1},u_{n})\le
\frac{1+b_{n-1}}{b_n}\max(u_{n-2},u_{n-1}),
$$
and replacing $n$ by $n-1$ in the first inequality
$$
 u_{n}\le \frac{1+b_{n-2}}{b_{n-1}}\max(u_{n-2},u_{n-1}).
$$
Combining the last two inequalities gives
$$
\max(u_n,u_{n+1})\le \max\left (\frac{1+b_{n-2}}{b_{n-1}},\frac{1+b_{n-1}}{b_{n}}\right )\,\max(u_{n-2},u_{n-1}),\quad n\ge 2,
$$
which implies the conclusion because $u_0=u_1=1.$
\end{proof}

 \begin{lem}\label{thm:lambda} Let $(b_n)$ and $(u_n)$ be as in
   Proposition~\ref{thm:pni2}. Then the sequence of eigenvalues
   $(\la_N)$ from \eqref{eq:lambda} satisfies 
$$\lambda_N\ge \left (\sum_{k=0}^N u_k^2\right )^{-1}.$$
\end{lem}
 \begin{proof}
By \cite[(1.12)]{B:C:I} we have 
$$
 \lambda_N\ge \left (\frac{1}{2\pi}\int_0^{2\pi} \sum_{k=0}^N|P_k(e^{it})|^2\,dt\right )^{-1}.$$
The conclusion follows now by Lemma~\ref{thm:pni}, which shows that $|P_k(e^{it})|\le u_k.$ 
\end{proof}
Using the assumption of Proposition~\ref{thm:pni2}, we adopt the notation
\begin{equation}\label{eq:eta}
1-\eta_k=\max\left (\frac{1+b_{2k-2}}{b_{2k-1}},\frac{1+b_{2k-1}}{b_{2k}} \right ),\quad k\ge 1.
\end{equation}
 
\begin{prop}\label{thm:pni3}  Let $(b_n)$ and $(u_n)$ be as in
   Proposition~\ref{thm:pni2}. Then the sequence of eigenvalues
   $(\la_N)$ from \eqref{eq:lambda} satisfies 
$$\lambda_{2N+1}\ge \left (2+2\sum_{k=1}^N\prod_{l=1}^k(1-\eta_l)^2\right )^{-1}.  $$
\end{prop}
\begin{proof}
By Lemma~\ref{thm:lambda} and the fact that $u_0=u_1=1$ we have
$$
 \lambda_{2N+1}\ge \left (2+\sum_{k=1}^N (u_{2k}^2+u_{2k+1}^2)\right
)^{-1}.
$$
 Proposition~\ref{thm:pni2} states that
$$
 \max(u_{2k},u_{2k+1})\le \prod_{l=1}^k (1-\eta_l).
$$
These two inequalities give the conclusion.
\end{proof}

\begin{lem}\label{thm:bn} Let $(b_n)$ be as in  Proposition~\ref{thm:pni2}
and define $\xi_n$ by
$$
\frac{b_{n-1}+1}{b_n}=1-\xi_n,\quad n\ge 1.
$$
Then
\begin{equation}\label{eq:formula}
b_n= \prod_{k=1}^n(1-\xi_k) ^{-1}\left
  [2+\sum_{k=1}^{n-1}\prod_{l=1}^k(1-\xi_l)\right ],\quad n\ge 1.
\end{equation}
\end{lem}

\begin{proof} We have
$$
b_n=(1-\xi_n)^{-1}(1+b_{n-1})=(1-\xi_n)^{-1}\left(1+(1-\xi_{n-1})^{-1}(1+b_{n-2})\right)=\ldots,
$$
and after $n$ steps the formula ends using $b_0+1=2$.
\end{proof}

\begin{thm}\label{thm:slow} Let $(\tau_n)$ be a  decreasing sequence of
  positive numbers satisfying $\tau_n\to 0$ and $\tau_0<1$.
Then there exists a determinate symmetric probability measure $\mu$ on
$\mathbb R$ for which $\la_N\ge \tfrac12\tau_N$ for all $N$.

In other words, the eigenvalues $\la_N$ can decay arbitrarily slowly.
\end{thm}

The proof depends on the following 

\begin{lem}\label{thm:concave}
Let $(e_n)$ be an increasing sequence of positive numbers such that
$e_0>1$ and $\lim e_n=\infty$.
 There exists 
a strictly increasing concave sequence $(d_n)$ such that
$d_0=1$, $d_n\le e_n$ for all $n$ and $\lim d_n=\infty$.  
\end{lem}

\begin{proof}
Define a function $f(x)$ on $[0,\infty)$ by $f(0)=e_0$ and
$f(x)=e_n$ for $n-1<x\le n,$ for $n\ge 1.$ This function is left
continuous. The discontinuity points in $]0,\infty[$  are  
denoted by $e_{n_k}$ for a strictly  increasing subsequence $n_k$ of natural numbers. Consider the sequence $A_k$ of points in the plane
given by $A_0$=$(0,1)$ and $A_k=(n_k,e_{n_k})$ for $k\ge 1.$  If we connect every two consecutive points $A_k$ and $A_{k+1}$
by the line segment we will
obtain a graph of a strictly increasing piecewise linear function $g(x)$ such that $g(x)\le f(x).$ Moreover
$g(x)$ tends to infinity at infinity. We are going to construct the graph of a concave function $h(x)$ such that
$h(x)\le g(x),$ $h(0)=1$ and $h(x)\to \infty$ as $x\to \infty.$ Once it is done the sequence $d_n=h(n)$ satisfies
the conclusion of the lemma. We will construct the graph of $h(x)$ by tracing the graph $\Gamma$ of $g(x).$ The points of   $\Gamma$ where the slope changes  will be called nodes.

We start at the point $(0,1)$ and draw a graph of the function $h(x)$. We go along the first line  segment of $\Gamma$  
 until we reach the first node. Then we inspect the slope of the next line segment of $\Gamma.$ If it is smaller than the slope
of the previous segment we continue along $\Gamma $ until we reach the next node. Otherwise we do not change slope
and continue drawing the straight line (below $\Gamma$). In this case 
  two possibilities may occur. The line   does not hit   $\Gamma.$ Then the graph of $h(x)$ is constructed.
Otherwise the line hits $\Gamma.$ Then  two cases are considered. If the line hits
a node of $\Gamma,$   then we follow the procedure described above for the first node. 
 If the line hits an interior point of a segment $\gamma$ of $\Gamma$, then we continue along
the segment $\gamma$ until we reach the next node, where we follow the procedure
described for the first node. We point out that the slope of the
segment $\gamma$ is necessarily strictly smaller than the slope of the
straight line followed before hitting $\gamma$.

In this way   a graph of $h(x)$ with the required properties is constructed. Observe that if the graph of $h(x)$ has
infinitely many points in common with $\Gamma$, then clearly $h(x)\to \infty $ as $x\to \infty.$ But if
there are only finitely many points in common with $\Gamma$, then $h(x)$ is eventually linear with a positive slope, hence
  $h(x)\to \infty $ as $x\to \infty.$ 

\end{proof}

\noindent{\it Proof of Theorem~\ref{thm:slow}}. Defining $e_n=1/\tau_n$,
there exists by Lemma~\ref{thm:concave} a
 concave, strictly increasing sequence $(d_n)$ with $d_0=1$ and $\lim
 d_n=\infty$ and such that $d_n\le e_n$. Moreover, we may
assume that $d_n\le {n+1}$ by replacing $d_n$ by $\min(d_n,n+1 ).$ In this way
we may also assume that $d_2\le 3.$
This implies that there exists a decreasing sequence of positive
numbers $c_k,k\ge 1$ such that $c_1\le 1$ and
$$
d_{2n}= 1+2\sum_{k=1}^nc_k.
$$
In fact, we define 
$$
c_1=(d_2-1)/2,\quad c_n=(d_{2n}-d_{2n-2})/2,\quad n\ge 2,
$$ 
so $(c_n)$ is decreasing because $d_{2n}$ is concave.

Let the sequence $\eta_k$ be defined by
\begin{equation}\label{eq:etak}
1-\eta_1=\sqrt{c_1},\quad
1-\eta_k=\sqrt{\frac{c_{k}}{c_{k-1}}},\quad k\ge 2.
\end{equation}
Then $\eta_k\ge 0$ and  
\begin{equation}\label{eq:d}
d_{2n}=1+2\sum_{k=1}^n \prod_{l=1}^k(1-\eta_l)^2.
\end{equation}

Define the sequence $\xi_k$ by
$$
\xi_{2k-1}=\xi_{2k}=\eta_k, \quad k\ge 1.
$$
Inspired by formula \eqref{eq:formula} we finally define a positive sequence $(b_n)$ by $b_0=1$ and
$$
b_n= \prod_{k=1}^n(1-\xi_k)^{-1}\left
  [2+\sum_{k=1}^{n-1}\prod_{l=1}^k(1-\xi_l)\right ],\quad n\ge 1.
$$
 Then we get for $n\ge 1$

\begin{eqnarray*}
b_{2n}&\le&\prod_{k=1}^n(1-\eta_k)^{-2}\left
  [3+2\sum_{k=1}^{n-1}\prod_{l=1}^k(1-\eta_l)^2\right ]\\
&=& 2\frac{2+d_{2n-2}}{d_{2n}-d_{2n-2}}<2\frac{2+d_{2n}}{d_{2n}-d_{2n-2}},
\end{eqnarray*}
where we used formula \eqref{eq:d}. This gives
$$
\frac{1}{b_{2n}} > \frac{d_{2n}-d_{2n-2}}{2(2+d_{2n})},
$$
and since $d_{2n}$ tends to infinity we get
\begin{equation}\label{eq:Carl}
\sum_{n=1}^\infty \frac{1}{b_{2n}}=\infty.
\end{equation}
In fact, assuming the contrary we get
$$
\infty > \sum_{n=1}^\infty \frac{1}{b_{2n}} > \sum_{n=1}^\infty\frac{d_{2n}-d_{2n-2}}{2(2+d_{2n})},
$$
so there exists $N\in\mathbb N$ such that for all $p\in\mathbb N$
$$
\frac12\ge \sum_{n=N+1}^{N+p}\frac{d_{2n}-d_{2n-2}}{2+d_{2n}}> 
 \sum_{n=N+1}^{N+p}\frac{d_{2n}-d_{2n-2}}{2+d_{2N+2p}}=\frac{d_{2N+2p}-d_{2N}}{2+d_{2N+2p}},
$$
but the right-hand side converges to 1 for $p\to\infty$, which is a contradiction.

The positive sequence $(b_n)$ defines a  system of orthonormal
polynomials via \eqref{eq:sym3t}. The corresponding symmetric probability
measure is determinate by Carleman's theorem because of
\eqref{eq:Carl}. 
Moreover,  by Proposition~\ref{thm:pni3} and formula \eqref{eq:d} we get
$$
\lambda_{2N}\ge \lambda_{2N+1}\ge \frac{1}{2d_{2N}}\ge
\frac{1}{2e_{2N}}=\frac12\tau_{2N}\ge\frac{1}{2}\tau_{2N+1}.\quad\square
 $$

\section{The indeterminate case}

\noindent Let $(s_n)$ be the  moment sequence \eqref{eq:HMP}. The inequality
$$
\sum_{n,m=0}^N s_{n+m}a_n\overline{a_m}\ge
c\sum_{k=0}^N |a_k|^2,\quad a\in\mathbb C^{N+1}
$$ 
can be rewritten
\begin{equation}\label{eq:fd1}
\int\left | \sum_{k=0}^N a_kx^k\right |^2\,d\mu(x)\ge
c\sum_{k=0}^N |a_k|^2.
\end{equation}

If we write
 $$
 \sum_{k=0}^N a_kx^k= \sum_{n=0}^N c_nP_n(x)
$$
and use \eqref{eq:OP1}, then \eqref{eq:fd1} takes the form
$$
\sum_{n=0}^N |c_n|^2\ge c \sum_{k=0}^N \left|\sum_{n=k}^N
  b_{k,n}c_n\right|^2.
$$
This immediately gives the following result:

\begin{lem}\label{thm:utm} The eigenvalues $\la_N$ are bounded below
  by a constant $c>0$ if and only if the upper triangular matrix
$\mathcal B=(b_{k,n})$ given by \eqref{eq:utmB}
corresponds to  a bounded operator on $\ell^2$ of norm $\le 1/\sqrt{c}$.
\end{lem}

Recalling that the indeterminate case was characterized in
\cite{B:C:I} by $\la_N$ being bounded below by a positive constant, we
see that the indeterminate case is characterized by the boundedness of
the operator $\mathcal B$. For a characterization of the lower
boundedness of $\la_N$ in a more general setting see \cite{B:D}.
 As noticed in \cite[Remark, p. 72]{B:C:I}, the indeterminacy is also
 equivalent to the boundedness of the matrix $\mathcal K$, cf. \eqref{eq:kappa}, which is
 automatically in trace class if it is bounded. 

Concerning the matrices $\mathcal A,\mathcal K,\mathcal B$, given by 
\eqref{eq:ajk}, \eqref{eq:kappa},\eqref{eq:utmB} respectively, we have:

\begin{prop}\label{thm:matrix} Assume that $\mu$ is
  indeterminate. Then the following matrix equations hold 
\begin{enumerate}
\item[(i)] $\mathcal K=\mathcal B^*\mathcal B$,
\item[(ii)] $\mathcal A=\mathcal B\mathcal B^*$.
\end{enumerate}
$\mathcal A,\mathcal B,\mathcal K$ are of trace class and 
$$
\tr(\mathcal A)=\tr(\mathcal K)=\rho_0,
$$
where $\rho_0$ is defined in \eqref{eq:rho}.

Furthermore, the sequence 
\begin{equation}\label{eq:ck}
c_k=\sqrt{a_{k,k}}=\left(\sum_{n=k}^\infty |b_{k,n}|^2\right)^{1/2},
\end{equation}
satisfies
\begin{equation}\label{eq:adelta}
\lim_{k\to\infty} k\root k\of{c_k}=0,
\end{equation}
and the matrix $\mathcal A=(a_{j,k})$ has the following property
\begin{equation}\label{eq:lpfinite}
\sum_{j,k=0}^\infty |a_{j,k}|^\varepsilon <\infty
\end{equation}
for any $\varepsilon >0$.
\end{prop}

\begin{proof}
From \eqref{eq:OP1} we have
\begin{equation}\label{eq:OP2}
b_{k,n}=\frac{1}{2\pi i}\int_{|z|=r}P_n(z)z^{-(k+1)}\,dz
=r^{-k}\frac{1}{2\pi}\int_0^{2\pi} P_n(re^{it})e^{-ikt}\,dt.
\end{equation}
Consider $r=1.$ Then, by Parseval's identity we have
\begin{equation}\label{eq:parseval}
\sum_{n=0}^N\sum_{k=0}^n|b_{k,n}|^2=
\frac{1}{2\pi}\int_0^{2\pi} \sum_{n=0}^N|P_n(e^{it})|^2\,dt.
\end{equation}
 Therefore, in the indeterminate case the matrix $\mathcal B$ is
 Hilbert-Schmidt with Hilbert-Schmidt norm $\rho_0^{1/2}$, cf. \eqref{eq:rho}.
 Hence both $\mathcal B^*\mathcal B$ and $\mathcal B\mathcal B^*$ are of trace class with trace
 $\rho_0$. Formula (i) of Proposition~\ref{thm:matrix} is an immediate consequence of Parseval's
 identity.

We know that $K_N(z,w)$ defined in \eqref{eq:kernelpol} converges to $K(z,w)$, locally uniformly in $\mathbb C^2$,
hence
\begin{equation}\label{eq:ny}
 a_{j,k}^{(N)}=\sum_{n=\max(j,k)}^N
  b_{j,n}b_{k,n} \to a_{j,k}
\end{equation}
for each pair $(j,k)$. The series
$$
\sum_{n=\max(j,k)}^\infty b_{j,n}b_{k,n}= \sum_{n=0}^\infty  b_{j,n}b_{k,n}
$$
is absolutely convergent for each pair $(j,k)$ because $\mathcal B$ is
Hilbert-Schmidt, so \eqref{eq:ny} implies (ii). 

Defining 
\begin{equation}\label{eq:TR}
 c_k=||\mathcal B^*\delta_k||=\left(\sum_{n=k}^\infty |b_{k,n}|^2\right)^{1/2},
\end{equation} 
where $\delta_k,k=0,1,\ldots$ denotes the standard orthonormal basis
in $\ell^2$, we have the following estimate for $r>1$ using the Cauchy-Schwarz inequality
$$
\left(\sum_{k=0}^\infty c_k\right)^2\le
 \sum_{k=0}^\infty r^{-2k}\sum_{k=0}^\infty
r^{2k}c_k^2=\frac{r^2}{r^2-1}\sum_{k=0}^\infty r^{2k}\sum_{n=k}^\infty
|b_{k,n}|^2.
$$
However, by 
\eqref{eq:OP2} and by Parseval's identity
we have
\begin{equation}\label{eq:trzy}
\sum_{k=0}^\infty r^{2k}\sum_{n=k}^\infty |b_{k,n}|^2=\sum_{n=0}^\infty\sum_{k=0}^n r^{2k}|b_{k,n}|^2=\sum_{n=0}^\infty
\frac{1}{2\pi}\int_0^{2\pi} |P_n(re^{it})|^2\,dt.
 \end{equation}
Let now 
\begin{equation}\label{eq:P}
P(z)=\left (\sum_{n=0}^\infty |P_n(z)|^2\right )^{1/2},\quad
z\in\mathbb C.
\end{equation}
We finally get
$$
\sum_{k=0}^\infty c_k\le
\frac{r}{\sqrt{r^2-1}}\left(\frac{1}{2\pi}\int_0^{2\pi}|P(re^{it})|^2\,dt\right)^{1/2}<\infty,
$$
but since 
$$
\langle |\mathcal B^*|\delta_k,\delta_k\rangle\le
||\,|\mathcal B^*|\delta_k\,||=||\mathcal B^*\delta_k||
$$
this shows that $|\mathcal B^*|$ and hence $\mathcal B$ is of trace class. 

For a given $\varepsilon >0$ we have
$P(z)\le C_\varepsilon e^{\varepsilon |z|}$ by a theorem of M. Riesz,
cf. \cite[Th. 2.4.3]{Ak}, hence by \eqref{eq:TR} and \eqref{eq:trzy}
$$
\sum_{k=0}^\infty r^{2k}c_k^2\le C_\varepsilon^2 e^{2\varepsilon r}.
 $$
For $r=k/\varepsilon$ we get in particular
$$
\left(\frac{k}{\varepsilon}\right)^{2k}c_k^2\le
C_{\varepsilon}^2e^{2k},
$$
hence
$$
\limsup_{k\to\infty} k\root k\of{c_k}\le e\varepsilon,
$$
which shows \eqref{eq:adelta}.

Using $|a_{j,k}|\le c_jc_k$, it is enough to prove that
$\sum_{k=0}^\infty c_k^\varepsilon <\infty$ for  $0<\varepsilon
<1$, which is weaker than \eqref{eq:adelta}.
\end{proof}

For a sequence $\a=(\a_n)\in\ell^2$ we consider the function
\begin{equation}\label{eq:ent}
F_{\a}(z)=\sum_{n=0}^\infty \a_nP_n(z)=\sum_{n=0}^\infty\b_nz^n,
\end{equation}
which is an entire function of minimal exponential type because
$$
|F_{\a}(z)|\le ||\a||P(z),
$$
where $P(z)$ is given by \eqref{eq:P}. The following result is a
straightforward consequence of \eqref{eq:ent}.

\begin{prop}\label{thm:ent1} The sequence of coefficients $\b=(\b_n)$ of
the power series of $F_{\a}$ belongs to $\ell^2$ and is given by
$\b=\mathcal B\a$. The operator $\mathcal B:\ell^2\to\ell^2$ is one-to-one with dense
range $\mathcal B(\ell^2)$.
\end{prop}

For a compact operator $T$ on $\ell^2$ we denote by
$\sigma_n(T),n=0,1,\ldots$ the singular values of $T$ in decreasing
order, i.e.
\begin{equation}\label{eq:sing}
\sigma_n(T)=\min_{V\subset \ell^2, \dim
  V=n}\quad\max_{\|v\|=1, \ v\perp V} \|T v\|.
\end{equation}

\begin{thm}\label{thm:lambdaNn}  Assume that $\mu$ is indeterminate. Let 
$$
\la_{N}=\la_{N,0}\le\la_{N,1}\le\ldots\le\la_{N,N}
$$
denote the $N+1$ eigenvalues of $\mathcal H_N$ and let 
$$
\la_{\infty,n}=\lim_{N\to\infty}\la_{N,n}.
$$
For $0\le n\le N$ we have
\begin{equation}\label{eq:sing1}
\sigma_n(\mathcal A)=\sigma_n(\mathcal B^*)^2\ge
\frac{1}{\la_{\infty,n}}\ge \frac{1}{\la_{N,n}}
\end{equation} 
and
\begin{equation}\label{eq:lambdainfty}
\lim_{n\to\infty} n^2\root n\of{\sigma_n(\mathcal A)}=0,\quad
\lim_{n\to\infty}\frac{\root n \of{\la_{\infty,n}}}{n^2}=\infty.
\end{equation}
\end{thm}

\begin{proof} By \eqref{eq:sing} we get
$$
\sigma_n(\mathcal B^*)
\le \max_{\|v\|=1,\, v\perp \delta_0,\ldots,\delta_{n-1}} \|\mathcal B^*v\|.
$$
Let $\Pi_n$ denote the projection onto $\{\delta_0,\ldots,\delta_{n-1}\}^\perp.$
Thus by \eqref{eq:TR}
$$
\sigma_n(\mathcal B^*)\le \|\mathcal B^*\Pi_n\|\le \left
  (\sum_{k=n}^\infty c_k^2\right )^{1/2}.
$$
On the other hand, for $r\ge 1$ we have
$$
\sum_{k=n}^\infty c_k^2\le \sum_{k=n}^\infty c_k^2
\frac{(k!)^2r^k}{(n!)^2r^n} \le 
\frac{S(r)}{(n!)^2r^n},
$$ 
where 
$$
S(r):=\sum_{k=0}^\infty (k!c_k)^2r^k<\infty
$$
because of \eqref{eq:adelta} and $\root k\of{k!}\sim k/e$, which holds
by Stirling's formula. Therefore
$$
\sigma_n(\mathcal B^*)^2\le \frac{S(r)}{(n!)^2r^n},
$$
and since $\sigma_n(\mathcal B^*)=\sqrt{\sigma_n(\mathcal B\mathcal B^*)}$ we get
\begin{equation}\label{eq:sigman}
\sigma_n(\mathcal A)=\sigma_n(\mathcal B\mathcal B^*) \le \frac{S(r)}{(n!)^2r^n},\quad r\ge 1,
\end{equation} 
which proves the first assertion of \eqref{eq:lambdainfty}.

 Let $Pr_N$ denote the projection in $\ell^2$ onto
$\span\{\delta_0,\ldots,\delta_N\}$.
We then have
$$
(\mathcal BPr_N)(\mathcal BPr_N)^*=\mathcal BPr_N\mathcal B^*\le
\mathcal B\mathcal B^*,
$$
and therefore for $n\le N$
$$
\sigma_n(\mathcal B\mathcal B^*)\ge \sigma_n((\mathcal BPr_N)(\mathcal
BPr_N)^*)=\sigma_n(\mathcal B_N\mathcal B_N^*)=
\sigma_n(\mathcal H_N^{-1}),
$$
where the last equality follows by Theorem~\ref{thm:ABC}. The matrix
$\mathcal H_N^{-1}$ is positive definite, so its singular values are
the eigenvalues which are the reciprocals of the eigenvalues of
$\mathcal H_N$, i.e. $\sigma_n(\mathcal H_N^{-1})=1/\la_{N,n}$. This
gives \eqref{eq:sing1} and the second assertion in
\eqref{eq:lambdainfty} follows.
\end{proof}

\begin{thm}\label{thm:A} The trace class operator $\mathcal
  A:\ell^2\to\ell^2$ is positive with spectrum
$$
\sigma(\mathcal A)=\{0\}\cup\{\la_{\infty,n}^{-1}\mid
n=0,1,\ldots\}.
$$
\end{thm}
\begin{proof} We will consider $\mathcal A^{(N)}=(a_{j,k}^{(N)})$ and
  $\mathcal B_N$ as  finite rank
  operators on $\ell^2$ by adding zero rows and columns.
Clearly, $\mathcal B_N$ tends to $\mathcal B$ in the Hilbert-Schmidt
norm, and therefore $\mathcal A^{(N)}=\mathcal B_N\mathcal B_N^*$
tends to $\mathcal A=\mathcal B\mathcal B^*$ in the trace norm.

The result now follows since the spectrum of $\mathcal A^{(N)}$
consists of the numbers $\la_{N,n}^{-1},n=0,1,\ldots,N$, by Theorem~\ref{thm:ABC}.
\end{proof}

\section{The Stieltjes-Wigert polynomials}

For $0<q<1$ we consider the moment sequence $s_n=q^{-(n+1)^2/2}$ given
by
\begin{equation}\label{eq:SW1}
\frac{1}{\sqrt{2\pi\log(1/q)}}\int_0^\infty x^n\exp\left(-\frac{(\log x)^2}{2\log(1/q)}\right)\,dx.
\end{equation}
We call it the Stieltjes-Wigert moment sequence because 
Stieltjes proved that it is indeterminate (he considered the special
value $q=\tfrac12$) and Wigert \cite{Wig} found the corresponding orthonormal
polynomials
\begin{equation}\label{eq:SW2}
P_n(x;q)=(-1)^n\frac{q^{\frac{n}{2}+\frac{1}{4}}}{\sqrt{(q;q)_n}}\sum_{k=0}^n
\left[\begin{matrix}n\\k\end{matrix}\right]_q(-1)^kq^{k^2+\frac{k}{2}}x^k.
\end{equation}
Here we have used the Gaussian $q$-binomial coefficients
$$
\left[\begin{matrix}n\\k\end{matrix}\right]_q=
\frac{(q;q)_n}{(q;q)_k(q;q)_{n-k}},
$$
involving the $q$-shifted factorial
$$
(z;q)_n=\prod_{k=1}^n(1-zq^{k-1}),\quad z\in\mathbb
C,n=0,1,\ldots,\infty.
$$
We refer to \cite{G:R} for information about this notation and
$q$-series. We have followed the normalization used in Szeg\H{o}
\cite{Sz1}, where $s_0=1/\sqrt{q}$. The Stieltjes-Wigert moment problem has been extensively
studied in \cite {Chr} using a slightly different normalization.

\begin{lem}\label{thm:SW1} The double sum
$$
\sum_{n=0}^\infty\sum_{k=0}^\infty b_{j,n}b_{k,n}s_{k+l}
$$
is absolutely convergent for each $j,l\ge 0$ and
$$
|a_{j,k}|\le \frac{q^{j^2+k^2}}{(q;q)_j(q;q)_k(q;q)_\infty^2}.
$$
Moreover, $\mathcal
A\mathcal H=\mathcal H\mathcal A=I$.
\end{lem}

\begin{proof} We find
$$
|b_{j,n}b_{k,n}|=\frac{(q;q)_n}{(q;q)_j(q;q)_k(q;q)_{n-j}(q;q)_{n-k}}q^{n+j^2+k^2+\tfrac{j+k+1}{2}},
$$
hence for $j\ge k$
\begin{eqnarray*}
|a_{j,k}|&\le& \frac{q^{j^2+k^2+\tfrac{j+k+1}{2}}}{(q;q)_j(q;q)_k}\sum_{n=j}^\infty\frac{(q;q)_n}{(q;q)_{n-j}(q;q)_{n-k}}q^n
\\
&=&\frac{q^{j^2+k^2+\tfrac{j+k+1}{2}}}{(q;q)_j(q;q)_k}\sum_{p=0}^\infty\frac{(q;q)_{j+p}}{(q;q)_p(q;q)_{j-k+p}}q^{j+p}\\
&= &
\frac{q^{j^2+k^2+\tfrac{j+k+1}{2}+j}}{(q;q)_{j-k}(q;q)_k}\sum_{p=0}^\infty\frac{(q^{j+1};q)_p}{(q;q)_p(q^{j-k+1};q)_p}q^p\\
&\le&
\frac{q^{j^2+k^2}}{(q;q)_j(q;q)_k}\sum_{p=0}^\infty\frac{q^p}{(q;q)_p(q;q)_\infty}=
 \frac{q^{j^2+k^2}}{(q;q)_j(q;q)_k(q;q)_\infty^2},
\end{eqnarray*}
where we have used the $q$-binomial theorem
\begin{equation}\label{eq:qbin}
\sum_{n=0}^\infty
\frac{(a;q)_n}{(q;q)_n}z^n=\frac{(az;q)_\infty}{(z;q)_\infty},\quad |z|<1
\end{equation}
with $a=0,z=q$.
By symmetry the estimate holds for all pairs $j,k$. Since
$s_{k+l}=q^{-(k+l+1)^2/2}$ it is clear that the double sum is
absolutely convergent. 

By Lemma~\ref{thm:ABCinfty} we then have 
$$
I=\mathcal B(\mathcal B^*\mathcal H)=(\mathcal B\mathcal B^*)\mathcal H=\mathcal A\mathcal H,
$$ 
and we clearly have $\mathcal H\mathcal A=\mathcal A\mathcal H$. 
\end{proof}

From \eqref{eq:SW2} we get
\begin{equation}\label{eq:SW3}
 P_n(0;q)=(-1)^n\frac{q^{\tfrac{n}{2}+\tfrac14}}{\sqrt{(q;q)_n}},
\end{equation}
hence by \eqref{eq:qbin}
\begin{equation}\label{eq:SW4}
\sum_{n=0}^\infty P_n^2(0;q)=\sum_{n=0}^\infty \frac{q^{n+\tfrac12}}{(q;q)_n}=\frac{\sqrt{q}}{(q;q)_\infty}.
\end{equation}
The matrix $\mathcal K=(\kappa_{j,k})$ defined in \eqref{eq:kappa} is
given by
\begin{equation}\label{eq:kappaspec}
\kappa_{j,k}=(-\sqrt{q})^{j+k}\frac{\sqrt{q}}{\sqrt{(q;q)_j(q;q)_k}}\sum_{p=0}^{\min(j,k)}
\left[\begin{matrix}j\\p\end{matrix}\right]_q\left[\begin{matrix}k\\p\end{matrix}\right]_q
q^{2p^2+p},
\end{equation}
hence
\begin{equation}\label{eq:rhospec}
\rho_0=\sum_{k=0}^\infty
\kappa_{k,k}=\sqrt{q}\sum_{k=0}^\infty\frac{q^k}{(q;q)_k}\sum_{p=0}^k
\left[\begin{matrix}k\\p\end{matrix}\right]_q^2q^{2p^2+p},
\end{equation}
in accordance with \cite{B:C:I}, which also contains other expressions
for $\rho_0$. From \eqref{eq:SW4},\eqref{eq:upperbound} with $z_0=0$ and
\cite[Theorem 1.2]{B:C:I} we get
$$
1/\rho_0\le \lim_{N\to\infty}\la_N<\frac{(q;q)_\infty}{\sqrt{q}}.
$$
From the general theory we know that the Stieltjes-Wigert moment
sequence has an N-extremal solution $\nu_0$, which has the mass
$c=(q;q)_\infty/\sqrt{q}$ (=the reciprocal of the value in
\eqref{eq:SW4}) at 0. It is a discrete measure concentrated at the
zeros of the entire function
$$
D(z)=z\sum_{n=0}^\infty P_n(0;q)P_n(z;q).
$$  
It is also known by a result of Stieltjes, that the measure
$\tilde\mu=\nu_0-c\varepsilon_0$ is determinate,
cf. e.g. \cite[Theorem 7]{B:C}. The moment sequence $(\tilde s_n)$ of
$\tilde\mu$ equals the Stieltjes-Wigert moment sequence except for the
zeroth moment, i.e.
$$
\tilde s_n=\left\{\begin{array}{ll}
(1-(q;q)_\infty)/\sqrt{q} & \mbox{if $n=0$}\\
q^{-(n+1)^2/2} & \mbox{if $n\ge 1$},
\end{array}\right.
$$
and similarly the corresponding Hankel matrices $\mathcal
H$ and $\tilde{\mathcal H}$ differ only at the
entry $(0,0)$.

We shall prove
\begin{thm}\label{thm:SW2} The smallest eigenvalue $\tilde\la_N$
  corresponding to the measure $\tilde\mu$ tends exponentially to zero
  in the sense that there exists a constant $A>0$ such that
$$
\tilde\la_N\le Aq^N.
$$
\end{thm} 

The proof of Theorem~\ref{thm:SW2} depends on the quite remarkable
fact that it is possible to find an explicit formula for the
corresponding orthonormal polynomials which will be denoted $\tilde
P_n(x;q)$. It is a classical fact, cf. \cite[p.3]{Ak}, that the
orthonormal polynomials $(P_n)$ corresponding to a moment sequence
$(s_n)$ are given by the formula
\begin{equation}\label{eq:detpol} 
P_n(x)=\frac{1}{\sqrt{D_{n-1}D_n}}\det\begin{pmatrix} s_0 & s_1
  &\cdots &s_n\\ \vdots &\vdots&\ddots&\vdots\\
  s_{n-1}&s_n&\cdots&s_{2n-1}\\
1 & x&\cdots&x^n
\end{pmatrix},
\end{equation}
where $D_n=\det(\mathcal H_n)$. In this way Wigert calculated the
polynomials $P_n(x;q)$, and we shall follow the same procedure for
$\tilde P_n(x;q)$.
Writing
\begin{equation}\label{eq:swtilde}
\tilde P_n(x;q)=\sum_{k=0}^n \tilde b_{k,n}x^k,
\end{equation}
we have
\begin{thm}\label{thm:SW3} For $0\le k\le n$
\begin{equation}\label{eq:swtilde1}
\tilde b_{k,n}=\tilde C_n (-1)^k\left[\begin{matrix}n\\k\end{matrix}\right]_q
q^{k^2+\tfrac{k}{2}}\left[1-(1-q^k)(q^{n+1};q)_\infty\right],
\end{equation}
where
\begin{equation}\label{eq:swtilde2}
\tilde C_n=\frac{(-1)^n q^{\frac{n}{2}+\frac{1}{4}}}{\sqrt{(q;q)_n}\sqrt{(1-(q^n;q)_\infty)(1-(q^{n+1};q)_\infty)}},
\end{equation}
i.e.
\begin{equation}\label{eq:swtilde3}
\tilde b_{k,n}=b_{k,n}\frac{1-(1-q^k)(q^{n+1};q)_\infty}{\sqrt{(1-(q^n;q)_\infty)(1-(q^{n+1};q)_\infty)}},
\end{equation}
where $b_{k,n}$ denote the coefficients of $P_n(x;q)$. Moreover,
\begin{equation}\label{eq:swtilde4}
\tilde D_n=D_n(1-(q^{n+1};q)_\infty),
\end{equation}
where $D_n=\det\mathcal H_n,\;\tilde D_n=\det\tilde{\mathcal H}_n$.
\end{thm}

\begin{proof}
We first recall the Vandermonde determinant
\begin{equation}\label{eq:van}
V_n(x_1,\ldots,x_n)=\det\begin{pmatrix}
  1&1&\cdots&1\\x_1&x_2&\cdots&x_n\\\vdots&\vdots&\ddots&\vdots\\
x_1^{n-1}&x_2^{n-1}&\cdots&x_n^{n-1}\end{pmatrix}=\prod_{1\le i<j\le n}(x_j-x_i).
\end{equation}
For an $n\times n$-matrix $(a_{j,k}),\;j,k=1\ldots n$ with non-zero
elements in the first row and column we have
$$
\det(a_{j,k})=\left(\prod_{j=1}^n a_{j,1}\right)\left(\prod_{k=1}^n
  a_{1,k}\right)\det(\frac{a_{j,k}}{a_{j,1}a_{1,k}}),
$$
and if $a_{j,k}=q^{-(j+k-1)^2/2},\,j=k=1,\ldots,n+1$, where $0<q<1$, we get in particular
\begin{equation}\label{eq:Dn1}
D_n=\det(q^{-(j+k-1)^2/2})=\left(\prod_{j=1}^{n+1}q^{-j^2/2}\right)^2\det(q^{-(j-1)(k-1)+1/2}),
\end{equation}
hence using $S_n=\sum_{j=1}^n j^2=n(n+1)(2n+1)/6$
\begin{equation}\label{eq:Dn2}
D_n=q^{-S_{n+1}+(n+1)/2}V_{n+1}(1,q^{-1},\ldots,q^{-n}).
\end{equation}
By \eqref{eq:van} we get
$$
V_{n+1}(1,q^{-1},\ldots,q^{-n})=\prod_{i=0}^n\prod_{j=i+1}^n\frac{1}{q^j}(1-q^{j-i})=
\prod_{i=0}^n q^{-(n-i)(n+i+1)/2}(q;q)_{n-i},
$$
and after some reduction
\begin{equation}\label{eq:Dn3}
V_{n+1}(1,q^{-1},\ldots,q^{-n})=q^{-S_n}\prod_{j=1}^n(q;q)_j.
\end{equation}
We denote by $A_{r+1,p+1}$ respectively $\tilde{A}_{r+1,p+1}$ the cofactor of the entry $(r+1,p+1)$ of the
Hankel matrix 
$\mathcal H_n=(q^{-(j+k-1)^2/2})$
 respectively $\tilde{\mathcal H}_n$,  where $r,p=0,1,\ldots,n$.
When $r=0$ or $p=0$ we clearly have $A_{r+1,p+1}=\tilde{A}_{r+1,p+1}$.
 For
 $0<p<n$ we get
$$
A_{n+1,p+1}=(-1)^{n-p}\det\left(q^{-(j+k-1)^2/2}\;|\;\stackrel{\mbox{\scriptsize{$j=1,\ldots,n$}}}{\mbox{
\scriptsize{$k=1,\ldots,n+1;k\neq
    p+1$}}}\right)
$$
$$
= (-1)^{n-p}\prod_{j=1}^n q^{-j^2/2}\prod_{\stackrel{k=1}{k\neq
  p+1}}^{n+1}q^{-k^2/2}
\det\left(q^{-(j-1)(k-1)+1/2}\;|\;\stackrel{\mbox{\scriptsize{$j=1,\ldots,n$}}}{\mbox{
\scriptsize{$k=1,\ldots,n+1;k\neq
    p+1$}}}\right)
$$
$$
=(-1)^{n-p}q^{-S_{n+1}+((n+1)^2+(p+1)^2+n)/2}V_n(1,q^{-1},\ldots,q^{-(p-1)},q^{-(p+1)},\ldots,q^{-n}).
$$
However,
\begin{eqnarray*}\lefteqn{V_{n+1}(1,q^{-1},\ldots,q^{-n})}\\
&=& V_n(1,q^{-1},\ldots,q^{-(p-1)},q^{-(p+1)},\ldots,q^{-n})\prod_{j=0}^{p-1}(q^{-p}-q^{-j})
\prod_{j=p+1}^{n}(q^{-j}-q^{-p})\\
&=& V_n(1,q^{-1},\ldots,q^{-(p-1)},q^{-(p+1)},\ldots,q^{-n})(q;q)_p(q;q)_{n-p}q^{-(n^2+p^2+n-p)/2},
\end{eqnarray*}
so we finally get
\begin{equation}\label{eq:Dn4}
A_{n+1,p+1}/D_n=(-1)^nq^{(n+1)(n+1/2)}\frac{(-1)^pq^{p(p+1/2)}}{(q;q)_p(q;q)_{n-p}},\quad
0<p<n.
\end{equation}
It can be verified that this formula also holds for $p=0$ and $p=n$.

Using \eqref{eq:detpol} it is now easy to verify formula \eqref{eq:SW2} for
the Stieltjes-Wigert polynomials $P_n(x;q)$.

Expanding after the first column we get
$$
\tilde D_n=D_n-cA_{1,1},\quad c=(q;q)_\infty/\sqrt{q},
$$
and a calculation as above leads to
$$
A_{1,1}=q^{-S_{n+2}+5+9(n/2)}V_n(1,q^{-1},\ldots,q^{-(n-1)}),
$$
which gives \eqref{eq:swtilde4}. Moreover, for $0<p\le n$ we find
$$
\tilde
A_{n+1,p+1}=A_{n+1,p+1}-c(-1)^{n-p}\det\left(q^{-(j+k+1)^2}\;|\;\stackrel{\mbox{\scriptsize{$j=1,\ldots,n-1$}}}
{\mbox{\scriptsize{$k=1,\ldots,n;k\neq p$}}}\right),
$$
and the last determinant can be calculated to be
$$
\frac{D_{n-1}}{(q;q)_{n-p}(q;q)_{p-1}}q^{-n^2-(n-1)/2+p(p+1/2)}.
$$
This leads to
\begin{equation}\label{eq:Dn5}
\frac{\tilde A_{n+1,p+1}}{\tilde
  D_n}=\frac{A_{n+1,p+1}}{D_n}\frac{1-(1-q^p)(q^{n+1};q)_\infty}
{1-(q^{n+1};q)_\infty}.
\end{equation}
It can be verified that this formula also holds for $p=0$ because of
\eqref{eq:swtilde4}, and it is now easy to establish \eqref{eq:swtilde3}.
\end{proof}

{\bf Proof of Theorem~\ref{thm:SW2}} By Lemma~\ref{thm:upperbound} we
get
$$
\tilde\la_N\le
(\tilde{P}_N(0;q))^{-2}=\frac{(q;q)_N(1-(q^{N+1};q)_\infty)(1-(q^N;q)_\infty)}{q^{N+1/2}}.
$$
From the power series expansion of the entire function $(z;q)_\infty$
we have
$$
1-(z;q)_\infty\sim \frac{z}{1-q},\quad z\to 0,
$$
hence
\begin{equation}\label{eq:ass1}
1-(q^N;q)_\infty\sim \frac{q^N}{1-q},\quad N\to \infty,
\end{equation}
and therefore
$$
(\tilde{P}_N(0;q))^{-2}\sim
\frac{(q;q)_\infty}{(1-q)^2}q^{N+1/2},\quad N\to\infty,
$$
which proves the statement of the theorem.

 \begin{rem} {\rm The measure $\tilde{\mu}$ is determinate of index 0,
     cf. \cite{B:D0}, so by Corollary 2.1 in \cite{B:D} we know that
     the next smallest eigenvalue $\tilde{\la}_{N,1}$ of
     $\tilde{\mathcal H}_N$ is bounded below.}
\end{rem}

\noindent
Christian Berg\\
Department of Mathematics, University of Copenhagen,
Universitetsparken 5, DK-2100, Denmark\\
e-mail: {\tt{berg@math.ku.dk}}

\vspace{0.4cm}
\noindent
Ryszard Szwarc\\
Institute of Mathematics,
University of Wroc{\l}aw, pl.\ Grunwaldzki 2/4, 50-384 Wroc{\l}aw, Poland 
\newline  and
\newline \noindent Institute of Mathematics and Computer Science, 
University of Opole, ul. Oleska 48,
45-052 Opole, Poland\\
e-mail: {\tt{szwarc2@gmail.com}}

\end{document}